\documentclass[12pt]{amsart}

\usepackage{latexsym,amsfonts,amsmath,amssymb,amsthm,url,graphics,psfrag,amscd,stmaryrd}
\usepackage[english]{babel}
\usepackage[latin1]{inputenc}
\usepackage[T1]{fontenc}
\usepackage{graphics}
\usepackage{graphicx}
\usepackage{color}

\usepackage{pstricks-add}

\newcommand{\E}{\mathbb{E}}
\newcommand{\Z}{\mathbb{Z}}

\newcommand{\pp}{\mathbb{P}}

\newcommand{\kA}{\mathcal{A}}
\newcommand{\kB}{\mathcal{B}}
\newcommand{\kC}{\mathcal{C}}

\newcommand{\kO}{\mathcal{O}}

\newcommand{\kF}{\mathcal{F}}
\newcommand{\kG}{\mathcal{G}}

\newcommand{\kM}{\mathcal{M}}
\newcommand{\kE}{\mathcal{E}}

\newcommand{\kI}{\mathcal{I}}

\newcommand{\kL}{\mathcal{L}}

\newcommand{\pn}{\mathbb{P} }

\newcommand{\en}{\mathbb{E} }

\newcommand{\poi}{\textrm{Poi}}

\newcommand{\lf}{\lfloor}
\newcommand{\rf}{\rfloor}

\newcommand{\lin}{\left[\kern-0.15em\left[}
\newcommand{\rin} {\right]\kern-0.15em\right]}
\newcommand{\linf}{[\kern-0.15em [}
\newcommand{\rinf} {]\kern-0.15em ]}
\newcommand{\ilin}{\left]\kern-0.15em\left]}
\newcommand{\irin} {\right[\kern-0.15em\right[}

\newtheorem{lem}{Lemma}[section]

\newtheorem{prop}[lem]{Proposition}
\newtheorem{theo}[lem]{Theorem}

\newtheorem*{ack}{Acknowledgments}
\title[Contact process on  random geometric graphs ]
       {\bf Super-exponential extinction time of  the contact process on  random geometric graphs }
\author{Van Hao Can}
\address{Aix Marseille Universit\'e, CNRS, Centrale Marseille, I2M, UMR 7373, 13453 Marseille, France}
\address{Institute of Mathematics, Vietnam Academy of Science and Technology, 18 Hoang Quoc Viet, 10307 Ha Noi, Viet Nam}
\email{cvhao89@gmail.com}

 \keywords{Contact process; extinction time ;  random geometric graph.
} 
\subjclass[2010]{82C22; 60K35; 05C80.}

\begin{document}
\maketitle
\begin{abstract}
In this paper, we prove  lower and upper bounds for the extinction time of the contact process on  random geometric graphs with connection radius  tending to infinity. We obtain that for any infection rate $\lambda >0$, the  contact process on these graphs survives a time super-exponential in the number of vertices.
\end{abstract}

\section{Introduction}
We will study the contact process on  random geometric graphs (RGGs) in $d \geq 2$ dimensions  with intensity function $g=g_n(x)$ and  connection  radius  $R$, denoted by $G(n,R,g)$.

\vspace{0.2 cm}
 A RGG  is constructed as follows.  The vertex set is composed of the atoms of  a Poisson point process with intensity $g$ on $[ 0, \sqrt[d]{n}]^d$.  Then for any two vertices  $v  \neq w $, we draw an edge between them  if $\|v-w \| \leq R $, where $\|\cdot\|$ denotes the Euclidean norm in $\mathbb{R}^d$. We will assume throughout this paper  that there are positive constants $b$ and $B$, such that
\begin{align} \label{cdg}
0< b  \leq g(x) \leq B < + \infty  \quad \textrm{ for all } x,
\end{align}
 here and below,  we remove the subscript $n$ in the function $g_n$ for simplicity.

\vspace{0.2 cm}
 The contact process is one of the most studied interacting particle systems and is also often interpreted as a model to describe the spread of a virus in a network (see  for instance \cite{L}). Mathematically, it can be defined as follows: given a locally finite graph $G=(V,E)$ and $\lambda >0$, the contact process on $G$ with infection rate $\lambda$ is a pure jump Markov process $(\xi_t)_{t\geq 0}$ on $\{0,1\}^V$. Vertices of $V$ (also called sites) are regarded as individuals which are either infected (state $1$) or healthy (state $0$). By considering $\xi_t$ as a subset of $V$ via $\xi_t \equiv \{v: \xi_t(v)=1\},$ the transition rates are given by
\begin{align*}
\xi_t \rightarrow \xi_t \setminus \{v\} & \textrm{ for $v \in \xi_t$ at rate $1,$ and } \\
\xi_t \rightarrow \xi_t \cup \{v\} & \textrm{ for $v \not \in \xi_t$ at rate }  \lambda |\{w \in \xi_t : \{v,w\} \in E \}|,
\end{align*}
where $|A|$ is the cardinality of a set $A$.

Originally the contact process was studied on integer lattices or homogeneous trees. More recently, probabilists started  investigating this process on some families of random networks  like  configuration models, or  preferential attachment graphs, see for instance \cite{ BBCS, C,CD, CS, MMVY, MVY}.

\vspace{0.2 cm}
Random geometric graphs have been extensively studied for a long time by many authors, see in particular Penrose's book \cite{P}. Recently, these graphs have also been considered as models of wireless networks (see e.g. \cite{GK}). Therefore,  there has been interest in  processes occurring on them, including the contact process in both theoretical and practical approaches, see for  example  \cite{FSS,G,PJ}.

In this paper, we are in particular interested in the extinction time of the contact process,
\begin{align*}
\tau_n = \inf \{t: \xi^{1}_t = \varnothing\},
\end{align*}
where $(\xi^1_t)$ is the contact process on $G(n,R,g)$ starting with all nodes infected. 

 It has been shown that the contact process on  finite graphs dies out almost surely, thus  $\tau_n < \infty $ a.s. Now,  it is interesting to determine the order of  magnitude of $\tau_n$. For  sparse graphs, i.e.\  graphs in which the number of edges is of order the number of vertices, we will show that the extinction time is at most exponential in the number of vertices (see Lemma \ref{bgg}). On the other hand, we will show in Section 2.1 that the extinction time of the contact process on a complete graph is super-exponential in the number of vertices.  For the random geometric graph $G(n,g,R)$, we observe that w.h.p.\ the number of vertices   is $\Theta(n)$ and the graph locally looks like a complete graph (all vertices in a ball of radius $R/2$ form a clique).  Hence, we can expect that $\log \tau_n$ is super-linear in $n$  as in  the case of the complete graph. (Note that there are graphs which are not sparse but for which $\log \tau_n = \kO(n)$, for example the configuration model with infinite mean degree, see Theorem 1.2 (ii) in \cite{CS}). 

In \cite[Theorem 1.2]{G},  Ganesan considers the contact process on an equivalent model of $G(n,R,g)$ in $2$ dimensions. Translating to our model, Ganesan proves that if $R \rightarrow \infty$ and $R^2=\kO(\log n)$, then there exist  positive constants $c=c(\lambda)$ and $C=C(\lambda)$, such that w.h.p.\ $C n \log n \geq \log \tau_n \geq c n R^2/ \log n$. 

\vspace{0,2cm}
In our main result, we will  prove that in all dimensions larger than or equal to $2$,  for any $\lambda >0$ and for all $R$ large enough,  w.h.p.\ $\log \tau_n = \Theta( n \log (\lambda R^d))$. 
 \begin{theo} \label{trg} 
Let $d\geq 2$ and  $\tau_n$ be the extinction time of the contact process on the graph $G(n,R,g)$ with $g$ satisfying \eqref{cdg} starting from full occupancy. Then there exist positive constants $c$, $C$ and $K$ depending only on $d,b, B$, such that the following statements hold.
\begin{itemize}
\item[$(i)$] For any $R=R(n)$ and $\lambda=\lambda(n)$ satisfying  $n \geq R^d \geq K / ( \lambda \wedge 1)$, w.h.p.
 $$\tau_n \geq \exp(c n \log ( \lambda R^d) ) $$
 and 
 \[ \frac{\tau_n}{\en(\tau_n)}\ \mathop{\longrightarrow}^{(\mathcal L)}_{n\to \infty}  \  \kE(1), \]
with $\kE(1)$ an exponential random variable with mean one.
  \item[$(ii)$] For all $R>0$, w.h.p. 
 $$\tau_n \leq \exp(C n \log ( \lambda R^d)).$$
\end{itemize}   
 \end{theo}
Part (i) implies that when  $R$ tends to infinity, the contact process survives a time super-exponential in $n$ regardless the value of $\lambda$. We usually say that in this case the critical value of the infection rate  is zero. On the other hand, recently in \cite{MS} M\'enard and Singh show that when $R$ is {\it fixed}, there is a non-trivial phase transition of the contact process on  {\it infinite} random geometric graphs (i.e.\ the vertices  are  atoms of a Poisson point process on the whole space $\mathbb{R}^d$). More precisely, they prove that there exists a constant $\lambda_c >0$, such that if $\lambda < \lambda_c$, the contact process dies out a.s.\ whereas if $\lambda > \lambda_c$, it survives forever with positive probability. Moreover, in \cite[Section 5.3.3]{Ct}, by slightly improving some details in the proof of M\'enard and Singh, we show that $\lambda_c(R) =\Theta( R^{-d})$. 

It has been observed in many examples that the contact process on  a sequence of finite graphs, say $(G_n)$, converging locally to some limiting graph, say $G$, exhibits a phase transition at the same critical value of infection rate as on the limit $G$: in the sub-critical regime, the contact process on $G$ dies out a.s.\ (resp.\ the  extinction time $\tau_n$ of the process on $G_n$ is of order $\log (|G_n|)$), whereas in the super-critical regime, the contact process survives forever with positive probability (resp.\ $\log \tau_n $ is of order $|G_n|$), see for instance \cite{C,CD,CMMV,MMVY,MV}.  

Our theorem \ref{trg} (i) implies that in the  case of random geometric graphs, this phase transition also  holds in a ''highly'' super-critical phase, i.e.\ it holds when $\lambda > C \lambda_c$, with $C$  a positive constant  independent of $R$.

\vspace{0.2cm} 
We now make some comments on the proof of Theorem \ref{trg}. The proof of (i) consists of two main steps. First,  we  find in $G(n,R,g)$ a {\it key  subgraph} composed of $\lceil cnR^{-d} \rceil$ adjacent complete graphs, each of size $\lfloor cR^d \rfloor$, see Lemma \ref{p2}. The proof of this part is based on the existence of long paths in super-critical site percolation in $\Z^d$.  Secondly,  we study the extinction time of the contact process on this key subgraph.  This part is based on a comparison between the contact process and a super-critical oriented percolation, see Section 4.  The proof of (ii) follows from a quite general argument: the extinction time of the contact process on a graph $G=(V,E)$ is at most $\exp(C|V| \log (|E|/|V|))$, for some positive constant $C$. 

\vspace{0.2 cm}
The paper is organized as follows. In Section 2, we prove some preliminary results on the contact process on  complete graphs and the oriented percolation in two dimensions. In Section 3, we prove the existence of a key subgraph mentioned above. In Section 4, we study the contact process on this key subgraph. In Section 5, we conclude the proof of Theorem \ref{trg} by using results in the previous sections.  In the last section, we study some extensions: the case $d=1$ and the equivalent model  considered in \cite{G}. 

\vspace{0.2cm}
We now fix some notation. We call size of a graph $G$ the cardinality of its set of vertices and we denote it by $|G|$.  For $\mu >0$, we denote by $\poi(\mu)$  a Poisson random variable with mean $\mu$ and $\kE(\mu)$ an exponential random variable with mean $1/\mu$. For  $x >0$, we denote by $\lf x \rf$  (resp.\ $\lceil x \rceil$) the greatest (resp.\ least) integer less (resp.\ greater) than or equal $x$.   If $f $ and $g$ are two real functions, we write $f= \mathcal{O}(g)$ if there exists a constant $C>0,$ such that $f(x) \leq C g(x)$ for all $x ;$  $f =\Theta(g) $ if $f= \mathcal{O}(g)$ and $g= \mathcal{O}(f);$  $f=o(g)$ if $g(x)/f(x) \rightarrow 0$ as $x \rightarrow \infty$. 
The term  w.h.p.\ means with probability tending to $1$.  
\section{Preliminairies}
\subsection{Contact process on  complete graphs} We denote by $K_m$ the complete graph of size   $m$. Similarly to the results for the contact process on star graphs in \cite{CD,MVY}, we prove the following.
\begin{lem} \label{l5} Assume that $\lambda \leq 1$ and $m \lambda \geq 640$. Then  the following assertions hold.
\begin{itemize}
\item[(i)] Let $(\xi_t)$ be the contact process on  $K_m$. Then
$$\pp \left(\inf_{T_m/2 \leq t \leq T_m}|\xi_t | \geq m/4 \Bigm| |\xi_0| \geq m/4 \right) \geq 1-2T_m^{-1},$$
with $T_m= \exp(m \log(\lambda m)/16)$.
\item[(ii)] Let $K^1_m$ and $K^2_m$ be two disjoint complete graphs of size $m$, and $K_{m\text{-}m}$  be the graph formed by adding an edge between these two graphs. Let $(\xi_t)$ be the contact process on  $K_{m\text{-}m}$. Then
$$\pp \left(|\xi_{T_m} \cap K^2_m| \geq m/4 \, \Big | \, |\xi_0 \cap K^1_m| \geq m/4 \right) \geq 1-5T_m^{-1}.$$
\end{itemize}
\end{lem}
\begin{proof} Part (i)  follows from the following claims 
 \begin{align}
 \pp \left( \inf\limits_{0 \leq t\leq T_m}|\xi_t| \geq  m/4 \Bigm| |\xi_0| \geq  m/2 \right) & \geq 1 - T_m^{-1}, \label{lg1}\\
   \pp \Big( \exists \,t \leq T_m/2 :|\xi_t| \geq  m/2 \Bigm| |\xi_0| \geq m/4 \Big) & \geq 1 - T_m^{-1}. \label{lg2}
    \end{align}
First, we observe that  $|\xi_t| $  increases by $1$ with rate  $\lambda  |\xi_t| (m -|\xi_t|)$ and decreases by $1$ with rate $|\xi_t|$. Therefore, the skeleton of $(|\xi_t|)$ is a random walk $(U_r)$ trapped at $0$, which satisfies $U_0= |\xi_0|$ and
\begin{align*}
U_{r+1} = U_r +1 & \quad \textrm{ with probability } \, p_1= \frac{\lambda(m-U_r)}{\lambda(m-U_r)+1 }, \\
U_{r+1} = U_r +1 & \quad \textrm{ with probability } \, 1-p_1.
\end{align*}
We now prove \eqref{lg1}. Assume that $|\xi_0| \geq m/2$. Then $U_0 \geq m/2$. Moreover,  if  $U_r \in ( m/4, 3m/4)$ then $p_1 \geq \lambda m/( \lambda m+ 4)$. Hence, when $U_r \in ( m/4, 3m/4)$, it stochastically dominates a   random walk $(X_r)$ satisfying  $X_0 =m/2$ and 
\begin{align*}
X_{r+1} = X_r +1 & \quad \textrm{ with probability } \, \frac{ \lambda m}{\lambda m +4 }, \\
X_{r+1} = X_r -1 & \quad \textrm{ with probability } \, \frac{4}{\lambda m +4 }.
\end{align*}
 Then $\theta^{X_r}$ is a martingale, where 
$$\theta= \frac{4 }{ \lambda m}.$$
Let $q$ be the probability that $X_r$ goes below $ m/4 $ before hitting $3m/4$. It follows from the optional stopping theorem that 
$$q \theta^{ m/4 }+(1-q) \theta^{3 m/4} \leq \theta^{ m/ 2}. $$
Therefore using $\lambda m \geq 640$, we get
\begin{align} \label{eqq}
q \leq \theta^{m/4} = (4/ \lambda m)^{m/4} \leq  T_m^{-3}/(2m^2).
\end{align} 
Hence,  the random walk $(X_r)$ (and thus  $(|\xi_t|)$) makes  at least $ \lf m^2 T_m \rf $ upcrossings between  $ m/2$ and $3 m/ 4$ before hitting $ m/4 $ with probability larger than 
\begin{align} \label{eb1}
1 - \lf m^2 T_m \rf T_m^{-3}/(2m^2) \geq 1- T_m^{-1}/2.
\end{align} 
The law of the waiting time between two upcrossings of $(|\xi_t|)$ stochastically dominates  $\mathcal{E}(L)$,  with  $L=  \lambda \lf m/2 \rf (m- \lf m/2 \rf)+ \lf m/2 \rf $, the waiting time when $  |\xi_t|= \lf m/2 \rf$. 

Suppose that $(|\xi_t|)$ makes  more than $ \lf m^2T_m \rf $ consecutive upcrossings. Then the time that $(|\xi_t|)$ stays above $m/4$  stochastically dominates  $S$, the sum of $ \lf m^2 T_m \rf$ i.i.d.\ exponential random variables  with mean $1/L$. By applying Chebyshev's inequality, we get 
\begin{align} \label{eb2}
\mathbb{P}(S < \lf m^2T_m \rf/2L) \leq 4/(\lf m^2 T_m \rf) \leq T_m^{-1}/2.
\end{align}
Since $L \leq m^2/2$, we deduce \eqref{lg1} from \eqref{eb1} and \eqref{eb2}.

We now  prove \eqref{lg2}. Assume that $|\xi_0| \geq m/4$. Using a similar argument as for $(X_r)$, we get that  when $U_r \in (m/8,m/2)$, it stochastically dominates a  random walk $(Y_r)$ satisfying   $Y_0=  m/4$ and
\begin{align*}
Y_{r+1} & =  Y_r+1 \textrm{ with probability } \, p_2= \frac{ \lambda m}{ \lambda m +2} , \\
Y_ {r+1} & =Y_r-1 \textrm{ with probability } \, 1-p_2.
\end{align*}
Let us define $$\sigma_Y = \inf \{r: Y_r \geq  m/2\} \qquad \textrm{and} \qquad \tilde{\sigma}_Y = \inf \{r: Y_r \leq  m/8\}. $$
Then similarly to \eqref{eqq}, we have
\begin{align} \label{s1}
  \mathbb{P}(\tilde{\sigma}_Y < \sigma_Y)  \leq (2/\lambda m)^{m/8}\leq T_m^{-1}/3.
\end{align} 
Since $Y_r - (2p_2-1 )r$ is a martingale, it follows from the  optional stopping  theorem that
$$m/4= \mathbb{E}(Y_{\sigma_Y \wedge r}) - (2p_2-1 ) \mathbb{E}(\sigma_Y \wedge r)  \leq  m/2 - (2p_2-1 ) \mathbb{E}(\sigma_Y \wedge r) .$$
Therefore using $m \lambda \geq 640$, we get
$$ \E(\sigma_Y \wedge r)  \leq  \frac{m}{4(2p_2 -1)} \leq m/3.
$$
Letting $t$ go to infinity, we obtain
\begin{align*} 
\E(\sigma_Y) \leq m/3.
\end{align*} 
Thus using Markov inequality, we have
\begin{align} \label{s2}
\pp(\sigma_Y \geq mT_m) \leq \E(\sigma_Y)/mT_m \leq T_m^{-1}/3.
\end{align} 
Now, let us define 
$$\sigma = \inf \{t: |\xi_t| \geq  m/2\} \qquad \textrm{and} \qquad \tilde{\sigma}= \inf \{t: |\xi_t| \leq  m/8\}. $$
Then by  \eqref{s1}, 
\begin{align} \label{sg1}
\pp(\tilde{\sigma} < \sigma) \leq \mathbb{P}(\tilde{\sigma}_Y < \sigma_Y) \leq T_m^{-1}/3.
\end{align}
On the other hand, when $|\xi_t| \in (m/8,m/2)$  the waiting time at each stage  is an exponential random variable with mean less than $1/M$, with $M= \lambda \lf m/8 \rf (m- \lf m/8 \rf) + \lf m/8 \rf$, the mean of the waiting time when $|\xi_t|= \lf m/8 \rf$. Therefore
\begin{align*}
\sigma 1(\sigma < \tilde{\sigma}) \preceq \sum_{i=1}^{\sigma_Y} E_i,
\end{align*}
where $(E_i)$ is a sequence of i.i.d.\ exponential random variables with mean $1/M$ and  independent of $\sigma_Y$.  Hence
\begin{align} \label{sg2}
\pp(T_m/2  \leq \sigma < \tilde{\sigma}) & \leq \pp(\sigma_Y \geq mT_m) + \pp \left( \sum_{i=1}^{ \lf mT_m \rf} E_i\geq T_m/2 \right) \notag \\
& \leq 2T_m^{-1}/3.
\end{align}
Here, we have used \eqref{s2} to bound the first term, and for the second one we note that
 $$\en(E_i)=1/M \leq 64/( 7\lambda m^2) \leq 1/(70m),$$
 thus using a standard large deviation result, we get a bound for this term. Now, it follows from  \eqref{sg1} and \eqref{sg2} that 
\begin{align*}
  \pp \left(\sigma \geq  T_m/2 \right) \leq  T_m^{-1},
    \end{align*}
which proves \eqref{lg2}.

For (ii), let $u$ and $v$ be two vertices in $K^1_m$ and $K^2_m$ respectively, such that there is an edge between $u$ and $v$. Let $(\xi'_t)$ (resp.\ $(\xi''_t)$) be the contact process on $K^1_m$ (resp.\ $K^2_m$). By (i), we have
\begin{align} \label{cc1}
\pp \left(\xi'_{T_m} \neq \varnothing \, \Big| \, |\xi'_0| \geq m/4 \right) \geq 1- 2 T_m^{-1}. 
\end{align} 
We now claim that 
\begin{align} \label{cl3}
\pp\left(\exists \, t \leq m^2/2 : |\xi''_t| \geq m/4  \, \Big | \, \xi'_{m^2/2} \neq \varnothing \right) \geq e^{-m/4}.
\end{align}
To prove \eqref{cl3}, it amounts to show that 
\begin{align} \label{cl1}
\pp \left(\exists \, t \leq m/2 : |\xi''_t| \geq m/4  \, \Big | \,  |\xi''_0|=1 \right) \geq 2 e^{-m/4},
\end{align}
and 
\begin{align} \label{cl2}
\pp \left(\textrm{$v$ gets infected before } m^2/4 \, \Big | \, \xi'_{m^2/4} \neq \varnothing \right) \geq 1/2.
\end{align}
For \eqref{cl1}, observe that  when $|\xi''_t| \leq m/4$,  it increases by $1$ in the next stage with probability  
$$\frac{\lambda (m-|\xi''_t|)}{\lambda (m-|\xi''_t|)+ 1  } \geq \frac{3 \lambda m}{3 \lambda m+4} > 0.9 ,$$
as $\lambda m \geq 640$. Moreover,  the waiting time to the next stage is an exponential random variable with mean less than $1$. Therefore, the probability that in all the $ \lceil m/4 \rceil$ first stages, $|\xi''_t|$ increases and the waiting time  is less than $1$, is   larger than 
$$\left( 0.9(1- e^{-1}) \right)^{ \lceil m/4 \rceil} \geq 2 e^{-m/4},$$
which implies \eqref{cl1}. For \eqref{cl2}, we note that 
\begin{align*}
\{\xi'_{m^2/4}\neq \varnothing\} \subset \bigcap_{i=0}^{\lf m^2/8 \rf -1} \kE_i,
\end{align*}
 where
\begin{align*}
\kE_i=\{\exists\, v_i \in K^1_m: \xi'_{2i}(v_i)=1\}.
\end{align*}
We define
\begin{align*}
\kI_i =  \kE_i \cap  \{& \textrm{there is no recovery at $u_i$ in $[2i,2i+1]$ and there is an infection spread from }\\
&  \textrm{ $u_i$ to $u$ in $[2i,2i+1]$, there is no recovery at $u$ in $[2i,2i+2]$ and there is an  } \\ 
& \textrm{ infection spread from $u$ to $v$ in } [2i+1,2i+2] \}.
\end{align*}
If $u_i \equiv u$, we only consider the recovery in $u$ and the infection spread from $u$ to $v$. 
We see that if one of $(\kI_i)$ occurs then $v$ gets infected before $m^2/4$ and for any $i=0, \ldots, \lf m^2/8 \rf -1$
\begin{align} \label{kii}
\pp \left(\kI_i \, \big | \,  \cap_{j=0}^{i}\kE_j \right) \geq e^{-3}(1-e^{-\lambda})^2 \geq \lambda^2/(4e^3),
\end{align}
as $\lambda \leq 1$. Therefore, by using induction we have 
\begin{align} \label{cpp}
&\pp \left(\textrm{$v$ is not  infected before $m^2/4$}   \right) \notag \\
& \leq \pp \left( \left( \bigcup_{i=0}^{\lf m^2/8 \rf -1} \kI_i \right)^c\cap \left(\bigcap_{i=0}^{\lf m^2/8 \rf -1} \kE_i\right)\right)\notag\\
 &\leq (1-\lambda^2/(4e^3))^{\lf m^2/8 \rf } \notag \\ 
 &\leq 1/2,
\end{align}
since $\lambda m \geq 640$. Thus \eqref{cl2} follows. 

We now prove (ii) by using \eqref{cl3}. Suppose that $\xi'_{T_m} \neq \varnothing$. We divide the time interval $[0,T_m/2]$ into $ \lf T_m/m^2 \rf $ small intervals of length $m^2/2$. Then  \eqref{cl3} implies that in each interval with probability larger than $e^{-m/4}$, there is a time $s$, such that that $|\xi''_s| \geq m/4$. Hence, similarly to \eqref{cpp} we have  
\begin{align} \label{cc2}
\pp \left( \exists \, s \leq T_m/2 : |\xi''_s| \geq m/4 \Bigm | \xi'_{T_m/2} \neq \varnothing \right) \geq 1- (1-e^{-m/4})^{\lf T_m/m^2 \rf} \geq 1-  T_m^{-1}.
\end{align}
Suppose that $|\xi''_s| \geq m/4$ with $s \leq T_m/2$. Then  (i) implies that  $|\xi''_{T_m}| \geq m/4$ with probability larger than $1-2T_m^{-1}$.  Combining this with \eqref{cc1} and \eqref{cc2}, we get (ii).
\end{proof}
\subsection{Oriented percolation on finite sets} 
For any positive integer $\ell$, we consider an oriented percolation process on $[0,  \ell] $ with parameter $q$ defined  as follows. Let 
$$\Gamma = \{(i,k)  \in [ 0,  \ell] \times \mathbb{N} : i+k \textrm{ is even}\}.$$
For each pair of sites $(i,k)$ and $(j,k+1)$ with $j=i\pm 1$, we draw an arrow from $(i,k)$ to $(j,k+1)$ with probability $p$, all these events being independent. Given the initial configuration $A \subset [0, \ell]$, the oriented percolation  $(\eta_t)_{t \geq 0}$  is defined by
  $$\eta_t^A = \{i \in [0, \ell]  : \exists \, j \in A \textrm{ s.t.} \, \,  (j,0) \rightarrow (i,t) \} \textrm{ for } t \in \mathbb{N},$$
  where the notation $(j,0) \rightarrow (i,t)$ means that there is an oriented path from $(j,0)$ to $(i,t)$. If $A=\{x\}$, we simply write $(\eta_t^x)$. We call $(\eta_t)$ a {\it Bernoulli oriented percolation with parameter $q$}.
  
The oriented percolation on $\mathbb{Z} $, denoted  by $(\bar{\eta}_t)$, was investigated  by Durrett in \cite{D}. Using his results and techniques, we will prove the following.
   \begin{lem} \label{lc}
Let $(\eta_t)$ be the oriented percolation on $[ 0, \ell]$ with  parameter $q$.  Then there exist  positive constants $\varepsilon$ and  $c$, independent of $q$ and $\ell$, such that if $q \geq 1- \varepsilon$ then the following statements hold.
\begin{itemize}
\item[(i)] For any $\ell$, and $x \in [0, \ell]$
\begin{align*}
\pp \big ( \, \exists \, r, s \leq 2 \ell, \textrm{ s.t. } \eta^{x}_r (0)=1, \eta^{x}_s (\ell)=1 \big ) \geq c.
\end{align*}
\item[(ii)] For any $\ell$,
$$\pn ( \eta^{{\bf 1}}_{t_{\ell}} \neq \varnothing ) \geq 1 - 1/ t_{\ell},$$
  where 
  $t_{\ell}= \lf (1-q)^{-c\ell} \rf$ and $(\eta^{{\bf 1}}_t)$ is the oriented percolation starting with $\eta_0^{{\bf 1}}= [0, \ell]$.\\
  \item[(iii)] There exist a positive constant $\beta \in (0,1)$ and an integer  $s_{\ell}\in [\exp( c\ell), 2\exp(c \ell)] $, such that
$$\pn \Big( \big |\eta^{{\bf 1}}_{s_{\ell}} \cap [(1-\beta) \ell/2,(1 + \beta) \ell/2 ] \big |   \geq 3\beta \ell/4 \Big) \geq 1 - \exp(-c \ell).$$
\end{itemize}
  \end{lem}
\begin{proof}
Part (i) is similar to Theorem B.24 (a) in \cite{L} and (ii) can be proved using a contour argument as  in \cite[Section 10]{D}.

We now prove (iii). Let  $(\bar{\eta}_t)$ be the oriented percolation on $\mathbb{Z}$. Then   
\begin{align*}
\alpha= \pp(\bar{\eta_t}^{0} \neq \varnothing\, \forall \, t) \rightarrow 1 \quad \textrm{ as } \quad q \rightarrow 1.  
\end{align*}
Hence we can assume that $\alpha >3/4$. Now we define 
\begin{align*}
\ell_1= \lf (8- \alpha )\ell/16 \rf, \quad \quad \ell_2 = \lf (8+ \alpha )\ell/16 \rf \quad \textrm{ and } \quad \ell_3 = \lf \ell/4 \rf.
\end{align*}
We claim that there exists a positive constant $c$, such that for any $A \subset[\ell_1, \ell_2]$ with $|A| \geq 3 \alpha \ell/32$, 
\begin{align} \label{eal}
\pp \Big( \big |\eta^A_{\ell_3} \cap [\ell_1, \ell_2] \big | \geq 3 \alpha\ell/32 \Big) \geq 1 - \exp(-c \ell).
\end{align}
Suppose that $\eqref{eal}$ holds for a moment, we now prove (iii). Let  $A_0$ be an arbitrary subset of $[\ell_1, \ell_2]$ satisfying $|A_0| \geq 3 \alpha\ell/32$. Then we define
\begin{eqnarray*}
\kM_1= \big \{ \big |\eta^{A_0}_{\ell_3} \cap [\ell_1, \ell_2] \big | \geq 3 \alpha\ell/32 \big \}.
\end{eqnarray*}
It follows from \eqref{eal} that
\begin{equation} \label{kmm}
\pp(\kM_1) \geq 1- \exp(-c\ell).
\end{equation}
Moreover, if  $\kM_1$ happens,  $\eta^{A_0}_{\ell_3} \cap [\ell_1, \ell_2]$ contains a subset $A_1$  whose cardinality is  larger than $3 \alpha\ell/32$. Thus we can define 
\begin{eqnarray*}
\kM_2= \kM_1 \cap \big \{ \big |\eta^{A_1, \ell_3}_{2\ell_3} \cap [\ell_1, \ell_2] \big | \geq 3 \alpha\ell/32 \big \},
\end{eqnarray*}
where for all $0\leq s \leq t$ and $A \subset [0,\ell]$,
\[\eta^{A,s}_t= \big \{i \in [0, \ell]: \exists \, j \in A \, \textrm{ s.t. } \, (j,s) \rightarrow (i,t) \big \}.\]
Similarly, for all  $k\geq 2$ we define
\begin{eqnarray}
\kM_{k+1}= \kM_k \cap \big \{ \big |\eta^{A_k, k\ell_3}_{(k+1)\ell_3} \cap [\ell_1, \ell_2] \big | \geq 3 \alpha\ell/32 \big \},
\end{eqnarray}
where $A_k$ is a subset of $\eta^{A_{k-1}, (k-1)\ell_3}_{k\ell_3} \cap [\ell_1, \ell_2]$ satisfying $|A_k| \geq 3 \alpha \ell /32$. 

By \eqref{eal}, we have for all $k\geq 1$
\begin{equation*}
\pp \left(\kM_{k+1} \, \big | \, \kM_k \right) \geq 1 - \exp(-c \ell),
\end{equation*}
or equivalently 
\begin{equation} \label{kmk}
\frac{\pp(\kM_{k+1})}{\pp(\kM_k)} \geq 1 - e^{-c \ell}.
\end{equation}
Using \eqref{kmm} and \eqref{kmk}, we obtain that for $k_{\ell}=  \lf e^{c \ell /2}\rf$,
\begin{equation} \label{kmkl}
\pp(\kM_{k_{\ell}}) \geq  \left( 1- e^{-c \ell}\right)^{k_{\ell}} \geq 1 - 1/ k_{\ell}.
\end{equation}
We have 
\begin{equation} \label{kkl}
\kM_{k_{ \ell}} \subset \Big \{\big| \eta_{s_{\ell}}^{{\bf 1}} \cap [\ell_1, \ell_2] \big | \geq 3 \alpha \ell/32 \Big \},
\end{equation}
where $s_{\ell}= \ell_3 \times k_{\ell}$.  

On the other hand, if $\beta= \alpha/8$, then  $\ell_1 = \lf(1- \beta)\ell/2\rf $ and $\ell_2 = \lf (1+\beta) \ell/2 \rf$. Hence using \eqref{kmkl} and \eqref{kkl}, we get
$$\pn \Big( \big|\eta^{{\bf 1}}_{s_{\ell}} \cap [(1-\beta) \ell/2,(1 + \beta) \ell/2 ] \big| \geq 3\beta \ell/4 \Big ) \geq 1 - \exp(-c \ell/2),$$
which implies that (iii) holds with $\beta = \alpha/8$.

Now it remains to prove \eqref{eal}. First, we observe that for all $t \leq \ell_3$,
\begin{align*}
\bar{\eta}^{[\ell_1, \ell_2]}_{t} \subset [\ell_1 - t, \ell_2+ t] \subset [\ell_1 - \ell_3, \ell_2+ \ell_3] \subset [0, \ell],
\end{align*}
where $(\bar{\eta}_t)$ is the oriented percolation on $\Z$. Therefore, for any $A \subset [\ell_1, \ell_2]$
\begin{align*}
\left(\bar{\eta}^{A }_{t} \right)_{0 \leq t \leq \ell_3} \equiv  \left(\eta^{A}_{t} \right)_{0 \leq t \leq \ell_3}.
\end{align*}
Hence, to simplify notation, we use $(\eta_t)$ for the both processes in the interval $[0, \ell_3]$.  To prove \eqref{eal}, it suffices to show that there exists a positive constant $c$, such that
\begin{align}
\pp \Big(\big|\eta^x_{\ell_3} \cap [\ell_1, \ell_2]\big| \geq 3 \alpha \ell/32 \, \big| \, \eta^x_{\ell_3}  \neq \varnothing \Big)& \geq 1 -\exp(-c \ell) \textrm{ for all } x \in [\ell_1, \ell_2], \label{eal1}
\end{align}
\begin{align}
\pp \big(\eta^A_{\ell_3}  \neq \varnothing \big) & \geq 1 -\exp(-c \ell) \textrm{ for all } A \subset [\ell_1, \ell_2] \textrm{ with } |A| \geq 3\alpha \ell/32. \label{eal2}
\end{align}
To prove \eqref{eal1}, we define for any $A \subset \mathbb{Z}$ and $t \geq 0$
\begin{align*}
r^A_t: = \sup \{x: \exists y \in A, (y,0) \rightarrow (x,t)\} \\
l^A_t: = \inf \{x: \exists y \in A, (y,0) \rightarrow (x,t)\}.
\end{align*}
Then \eqref{eal1} is a consequence of the following claims. 
\begin{itemize}
\item[(a)] If $[\ell_1, \ell_2] \subset [l^x_{\ell_3}, r^x_{\ell_3}]$, then
$$\eta^{{\bf 1}}_{\ell_3} \cap [\ell_1, \ell_2] \equiv \eta^{x}_{\ell_3} \cap [\ell_1, \ell_2]. $$
\item[(b)] 
$\pp \Big( \big |\eta^{{\bf 1}}_{\ell_3} \cap [\ell_1, \ell_2] \big | \geq 3(\ell_2- \ell_1)/4 \Big) \geq 1 - \exp(-c \ell)$, as $3/4 <\alpha$.  \\
\item[(c)]
$\pp \left(  [l^x_{\ell_3}, r^x_{\ell_3}] \supset [\ell_1, \ell_2]  \Bigm| \eta^x_{\ell_3} \neq \varnothing \right) \geq 1 - \exp(-c \ell).$\\
\end{itemize}
We start with the claim (a). Suppose that $[\ell_1, \ell_2] \subset [l^x_{\ell_3}, r^x_{\ell_3}]$. Then there exists $y \leq \ell_1$ and $z \geq \ell_2$  together with two oriented paths: $\gamma_1$   from $(x,0)$ to $(y, \ell_3)$,  and   $\gamma_2$ from $(x,0)$ to $(z, \ell_3)$. Now, let $u$ be any element of $\eta^{{\bf 1}}_{\ell_3} \cap [\ell_1, \ell_2]$. Then $\ell_1 \leq u \leq \ell_2$ and there exists a vertex $v \in \Z$ and an oriented path $\gamma'$ from $(v,0)$ to $(u, \ell_3)$. The path $\gamma'$ is forced to intersect $\gamma_1$ or $\gamma_2$. In both cases, this implies  the existence of an oriented path from $(x,0)$ to $(u, \ell_3)$. Hence $u \in \eta^{x}_{\ell_3} \cap [\ell_1, \ell_2]$. We have just proved that 
\begin{equation*} \label{mnx}
\eta^{{\bf 1}}_{\ell_3} \cap [\ell_1, \ell_2] \subset \eta^{x}_{\ell_3} \cap [\ell_1, \ell_2].
\end{equation*}
The reverse is trivial,  hence  we obtain (a).

The  claim (b) follows from a result of Durrett and Schonmann    \cite[Theorem 1]{DS}. (Note that in \cite{DS}, the result is proved   for the contact process, but as mentioned by the authors the proof works just as well for  oriented percolation).  In fact, it still holds if we replace $3/4$ by  any $\alpha'< \alpha$. To prove (c), we observe that if $\eta^x_{\ell_3} \neq \varnothing$ then 
\begin{align*}
r^x_{\ell_3} = r^{(- \infty, x]}_{\ell_3}.
\end{align*} 
Moreover, by the main result of Section 11 in \cite{D}, there is a positive constant $c$, such that for all   integer $x$,
\begin{align*}
\pp \left(r^{(- \infty, x]}_{\ell_3} \leq x + \alpha \ell_3/2 \right) \leq \exp(-c \ell).
\end{align*}
Therefore if $x \in [\ell_1, \ell_2]$, then
\begin{align*}
\pp \left(r^{x}_{\ell_3} \geq \ell_2 \, \big | \, \eta^x_{\ell_3} \neq \varnothing  \right) \geq 1- \exp(-c \ell),
\end{align*}
since $x+ \alpha \ell_3/2 \geq \ell_1 + \alpha \ell_3/2 \geq \ell_2$. Similarly
\begin{align*}
\pp \left(l^{x}_{\ell_3} \leq \ell_1 \, \big | \, \eta^x_{\ell_3} \neq \varnothing \right)  \geq 1- \exp(-c \ell).
\end{align*}
Then the  claim (c) follows from the last two estimates.

Now we  prove \eqref{eal2} by using the same arguments as in Section 10 in \cite{D}. We say that $A$ is more spread out than $B$ (and write $A \succ B$) if there is an increasing function $\varphi$ from $B$ into $A$ such that  $|\varphi(x) - \varphi(y)| \geq |x-y|$ for all $x,y \in B$. (Note that this implies $|A| \geq |B|$).  In \cite{D}, Durrett proves that there is a coupling such that if $A \succ B$ then 
\begin{align*}
\eta^A_t \succ \eta^B_t \textrm{ for all } t \geq 0,
\end{align*} 
and as a consequence $|\eta^A_t| \geq |\eta^B_t|$ for all $t$. Hence 
\begin{align*}
\pp(\eta^A_t = \varnothing) \leq \pp(\eta^B_t = \varnothing).
\end{align*}
On the other hand, by (ii)
\begin{align*}
\pp \left(\eta^{[\ell_1, \ell_1+ \ell_4]}_{\ell_3} = \varnothing \right) \leq \exp(- c \ell),
\end{align*}
with $\ell_4 = \lf 3 \alpha \ell/32 \rf-1$. We observe that $A \succ [\ell_1, \ell_1+ \ell_4]$ for any $A$ with $|A| \geq 3 \alpha \ell/32$. Thus \eqref{eal2} follows from the last two inequalities.
\end{proof}

\section{Existence of a key subgraph}
In this section, by using a result on the existence of  long paths in a super-critical site percolation, we will show that the random geometric graph contains a key subgraph composed of $ \lceil cn/R^d \rceil $  complete graphs of size  $ \lf cR^d \rf$, with some $c>0$.

\vspace{0.2 cm}
%First, we recall the definition of a site percolation and state a result which we will prove below. 
 The Bernoulli site percolation on $\mathbb{Z}^d$ with parameter $p$ is defined as usual: designate each vertex in $\mathbb{Z}^d$  to be open independently with probability $p$ and closed otherwise. A path in $\mathbb{Z}^d$ is called open if all its sites are open.  Then there is a critical value $p_c^s(d) \in (0,1)$, such that if $p>p_c^s(d)$, then a.s.\ there exists an infinite open path (cluster), whereas if $p<p_c^s(d)$, a.s.\ there is no infinite cluster. 
\begin{lem} \label{lpp}
Consider the Bernoulli site percolation on $[0,n ] ^d$ with $d \geq 2$ and  $p > p^s_c(2)$. Then there exists a positive constant $\rho=\rho(p,d)$, such that w.h.p.\ there is an open path whose length is larger than $\rho n^{d}$. 
 \end{lem}
 
\vspace{0.2 cm} 
Now, we define the key subgraph. For $\ell,m \in \mathbb{N}$, we denote by $\kC(\ell, m)$  the graph obtained by glueing a  complete graph of size $ m$  to each vertex in a path of length $\ell$.  
\begin{lem} \label{p2}
Suppose that $d\geq 2$ and  $g$ satisfies \eqref{cdg}. Then there exist  positive constants $c$ and $C$, such that if $n \geq R^d \ge C$ then w.h.p.\  $G(n,R,g)$ contains as a subgraph a copy of $\kC( \lceil c n R^{-d} \rceil, \lf c R^d \rf)$.
\end{lem}
\begin{proof} If $n/R^d$ is bounded from above, then w.h.p.\ $G(n,R,g)$ contains a clique of size of order $n$ and thus the result follows. Indeed, by definition the vertices in $A=[0, R/\sqrt{d}]^d$  form a complete graph. Moreover  the number of vertices in $A$ is a Poisson random variable with mean $\int_{A} g(x) dx = \Theta( R^d ) = \Theta( n)$, and hence w.h.p.\ it is of order $n$.

We now assume that $n/R^d $ tends to infinity. Let $\ell = \lf \sqrt[d]{n}/(R/2\sqrt{d}) \rf$, we divide the box $[ 0, \sqrt[d]{n}]^d$ into $\ell^d$  smaller boxes of equal size, numerated by $(E_a)_{a \in [ 1, \ell ] ^d}$, whose side length  is $R/(2\sqrt{d})$. We see that if $v$ and $w$ are  in the same small box or in  adjacent ones, then $\|v-w \| \leq R$, hence these two vertices are connected. This implies that the vertices on a small box form a clique and  two adjacent cliques are connected. 

 For any  $a \in [1, \ell]^d$, let us denote by 
$$X_a = \# \{ v: v \in E_a\}$$ 
the number of vertices located in $E_a$. Then   $(X_a)$ are independent and  $X_a$ is a Poisson random variable with mean 
\begin{align} \label{mua}
\mu_a = \int_{E_a} g(x) dx \geq b \left( \frac{R}{2 \sqrt{d}} \right)^d =: \mu, 
\end{align}
since $g(x) \geq b$ for all $x$. For any $a$, we define
$$Y_a = 1(\{ X_a \geq \mu /2 \}).$$
Since $\pp(\poi(\mu) \geq \mu/2) \rightarrow 1$ as $\mu \rightarrow \infty$, it follows from \eqref{mua} that  $\pp(Y_a =1) \rightarrow 1$ as $R \rightarrow \infty$. Therefore  there is a positive constant $C$, such that if $R^d \geq C$, then 
\begin{align*}
\pp(Y_a =1) \geq p:= (1+p^s_c(2))/2.
\end{align*}
We note that the Bernoulli random variables $(Y_a)$ are independent. Hence  if we say the small box $E_a$ open when $Y_a=1$ and closed otherwise, then we get a site percolation on $[1,\ell]^d$ which stochastically dominates the Bernoulli site percolation on $[1,\ell]^d$ with parameter $p> p^s_c(2)$. Then  Lemma \ref{lpp} gives that w.h.p.\ there is an open path of length $ \rho \ell^d = \Theta( nR^{-d})$. On the other hand, in each open box, there is a clique of size $\mu/2 \Theta( R^d)$ and these cliques in adjacent open boxes are connected. Hence, the result follows by taking $c$ small enough.
 \end{proof}
\iffalse
 \subsection{Long path in a  site percolation.}  The Bernoulli site percolation on $\mathbb{Z}^d$ with parameter $p$ is defined as usual: designate each vertex in $\mathbb{Z}^d$  to be open independently with probability $p$ and closed otherwise. A path in $\mathbb{Z}^d$ is called open if all its sites are open.  Then there is a critical value $p_c^s(d) \in (0,1)$, such that if $p>p_c^s(d)$, then a.s.\ there exists an infinite open path (cluster), whereas if $p<p_c^s(d)$, a.s.\ there is no infinite cluster. 
 \fi
 
 \vspace{0.2 cm}
\noindent {\it Proof of Lemma \ref{lpp}}.
We set $m= \lf n^{1/4} \rf $. For $n,d \geq 2$, we say that the box $[0,n]^d$ is {\bf  $\rho$-good} if the site percolation cluster on it satisfies: 
 
there exist two vertices $x $ in $ \{m\} \times [0, n]^{d-1}$ and $y $ in $ \{n-m\} \times [ 0, n]^{d-1}$ and  an open path composed of  three parts: the first one  included in $[0,m] \times [ 0, n]^{d-1}$ has length larger than $m$ and ends at  $x$; the second one included in  $[m, n- m] \times [0, n]^{d-1}$ has length larger than $\rho n^{d}$, starts at $x$ and ends at $y$;  the third one  included in $[n-m,n] \times [ 0, n]^{d-1}$ starts at $y$ and has length larger than $m$.
  
  \vspace{0,2 cm}
 We now prove by induction on $d$ that if $p > p^s_c(2)$,  there is  a  positive constant $\rho_d=\rho(p,d)$, such that w.h.p.\ the box  $[0,n]^d$ is $\rho_d$-good. Then Lemma \ref{lpp} immediately follows. 

\vspace{0,2 cm}
 When $d=2$, the  statement is proved by Grimmett  in \cite[Theorem 1]{G2}.  We will prove it for $d=3$, the proof for $d\geq 4$ is exactly the same and will not be reproduced here.    
 
 \vspace{0,2 cm}
For $1 \leq i \leq n$, let $\Lambda_i = \{i\} \times [ 2 m, n- 2 m ]^2$. We define 
$$n_1= n-4m \quad \textrm{and} \quad m_1= \lf n_1^{1/4} \rf.$$
We say that the $i^{th}$ plane is  {\bf nice} (or $\Lambda_i$ is nice) if the site percolation on this plane satisfies:  $\Lambda_i$ is $\rho_2$-good (we consider $\Lambda_i$ as a box in $\mathbb{Z}^2$), and in each of the rectangles $\{i\} \times [m, 2m+m_1]\times [0,n]$ and $\{i\} \times [n- 2m-m_1,n-m]\times [0,n]$, there is a unique connected component of size larger than $m_1$, see Figure \ref{f1} for a sample of a nice plane.

The result for $d=2$ implies that w.h.p.\ $\Lambda_i$ is  $\rho_2$-good. On the other hand, we know that w.h.p.\ in the percolation on a box of size $n$ there is a unique open cluster having diameter  larger than $C \log n$ for some $C$ large enough (see for example Theorem 7.61 in \cite{G1}). Thus w.h.p.\ there is a unique open cluster of size larger $(C \log n)^d$.  Hence  $\Lambda_i$  is nice w.h.p.\ for all $i=1, \ldots, n$. Moreover, the events $\{ \Lambda_i \textrm{ is nice}\}$ are independent since the planes are disjoint. Therefore $\kA_n$ holds w.h.p.\ with 
\begin{align*}
\kA_n = \{ \# \{i: m \leq i \leq n-m,  \Lambda_i \textrm{ is nice}\} \geq n/2 \}.
\end{align*}
On $\kA_n$, there are more than $n/2$ disjoint open paths (they are in disjoint planes), each of which has length larger than $\rho_2n_1^2$. Thus, to obtain an open path of length of order $n^3$, we will glue these long paths using  shorter paths in  good boxes of nice planes. To do that, we define  

$\kB_n=\{$ for all $1 \leq i \leq \lf n/2 \rf $, there exist open paths: $\ell_{2i-1} \subset \{2i-1\} \times [ m, 2m] \times [0,n]$ whose  end vertices are  $u$ and $v$ with third coordinates $0$ and $n$ respectively; $\ell_{2i-1}' \subset \{2i-1\} \times [n- 2m,n-m] \times [0,n]$ whose end vertices are $u'$ and $v'$ with third coordinates  $0$ and $n$ respectively; $\ell_{2i} \subset \{2i\} \times [m,2 m] \times [0,n]$ whose end vertices are $z$ and $t$ with second coordinates $m$ and $2m$ respectively; $\ell_{2i}' \subset \{2i\} \times [n- 2m,n-m] \times [0,n]$ whose end vertices are $z'$ and $t'$ with second coordinates $n- 2m$ and $n-m$ respectively $\}$.

   We observe that  $\ell_{2i-1}$ is a bottom-top crossing and $\ell_{2i}$ is a left-right crossing in two consecutive rectangles. Then they intersect when we consider only the last two coordinates, and the same holds for $\ell'_{2i-1}$ and $\ell'_{2i}$. Hence on $\kB_n$, for all $1 \leq i \leq n-2$, there exist $a_i \in [m, 2m] \times [0,n]$ and $ b_i \in [n-2m, n-m] \times [0,n]$, such that 
\begin{align*}
(i,a_i) \in \ell_i \textrm{ and } (i+1,a_i) \in \ell_{i+1}, \\
(i,b_i) \in \ell_i' \textrm{ and } (i+1,b_i) \in \ell_{i+1}'. 
\end{align*}
In other word, we can jump from the $i^{th}$ plane to the next one in two ways. Moreover, on $\kA_n$  for all $i$ such that the $i^{th}$ plane is nice, the first part of the long open path in $\Lambda_i$ is connected to $\ell_i$ (as these paths  are in the same rectangle $\{i\} \times [m,2m+m_1] \times [0,n]$ and have length larger than $m_1$), and similarly the third part is connected to $\ell'_i$, see Figure \ref{f1}.
\begin{center}
\begin{figure}
\includegraphics[scale=0.5]{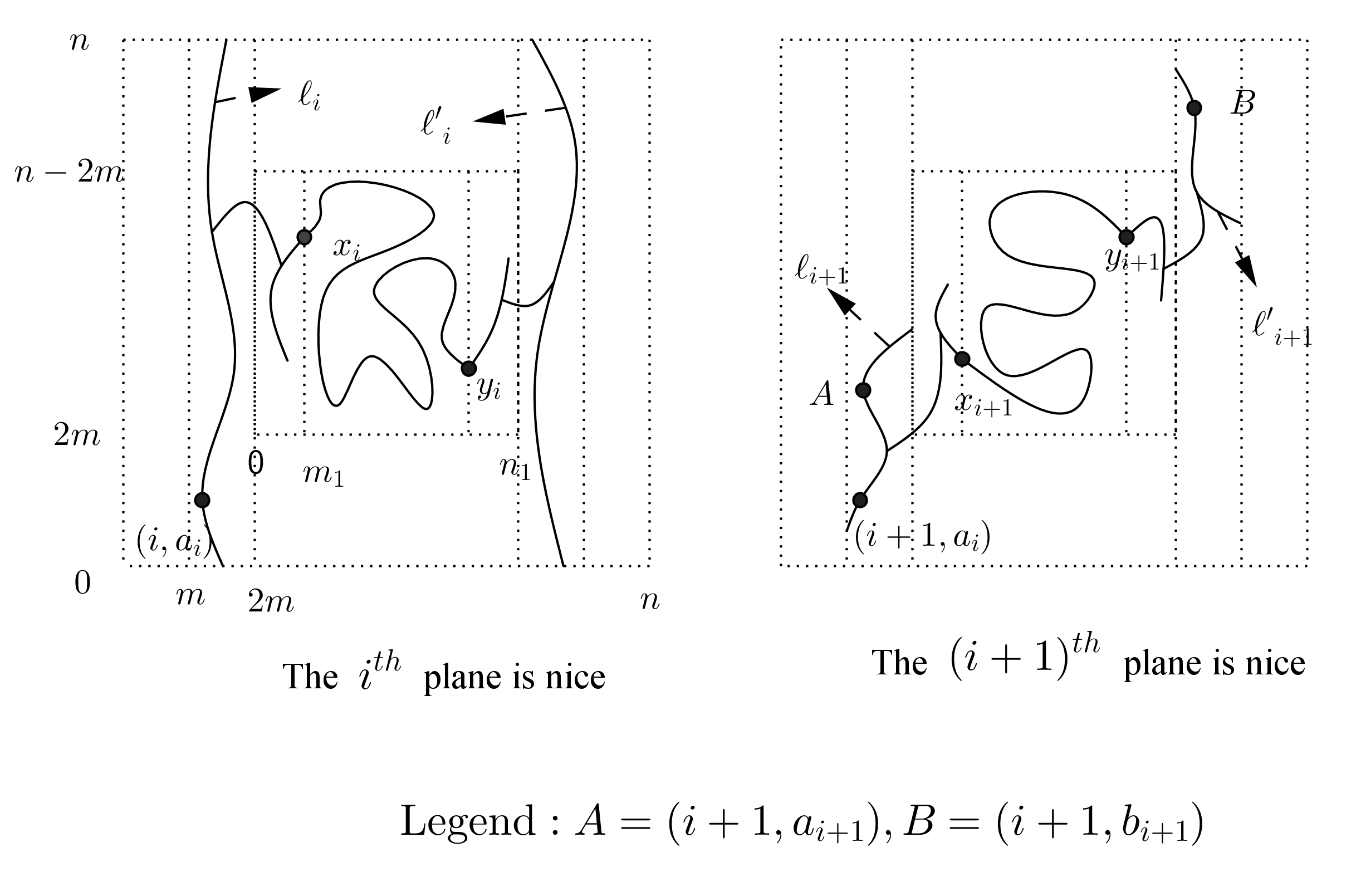}
\caption{Gluing two long paths.}
\label{f1}
\end{figure}
\end{center}
On $\kA_n \cap \kB_n$, we can find in $[m,n-m] \times [0,n]^2$ a path of length larger than $\rho_2 n^{3}/3$. Indeed, let $i$ be the first index, such that $i \geq m$ and $\Lambda_i$ is nice. We start at an end point, from the right for example, of the long path in $\Lambda_i$, then go along this long path towards the other end point. Then we can go to $\ell_i$ and arrive at $(i,a_i)$. Now we jump to $(i+1,a_i)$ (recall that it is a neighbor of $(i,a_i)$). If the $(i+1)^{th}$ plane is not nice, we go to $(i+1,a_{i+1})$ to jump to the next plane (note that both $(i+1,a_i)$ and $(i+1,a_{i+1})$ are in $\ell_{i+1}$). If the $(i+1)^{th}$ plane is nice, we now can touch and then go along to the long path in this plane and arrive at $(i+1,b_{i+1})$ to jump to the next plane. By continuing  this procedure, we can go through all the long paths of nice planes in the definition of $\kA_n$. The resulting path is in $[m,n-m] \times [0,n]^2$ and has length larger than $\rho_2 n^{3}/3$.
 
 Moreover, in the slabs $[0,m] \times [0,n]^2$ and $[n-m,n] \times [0,n]^2$, w.h.p.\ we can find two paths of length larger than $m$ which are connected to the long path we have just found above. These paths  form  the required three-parts long path. Therefore on $\kA_n \cap \kB_n$, w.h.p.\ the box $[0,n]^3$ is $\rho_3$-good with $\rho_3=\rho_2/3$.
  
Now it remains to show that $\kB_n$ holds w.h.p. We observe that the probability of the existence of such a path $\ell_i$ is larger than $1- \exp(-c m)$ for some $c>0$ (see for instance (7.70) in \cite{G1}). Thus $\kB_n$ holds w.h.p.

We summary here the change of proving the induction  from $d-1$  to $d$ when $d\geq 4$. First, in the definition of a {\it nice} box, we consider 
$$\Lambda_i = \{i \} \times [2m, n-2m]^{d-1},$$
and the uniqueness of the connected component of size larger than $m_1$ in the slabs $\{i\} \times [m,2m+m_1] \times [0,n]^{d-2}$ and $\{i\} \times [n-2m-m_1,n-m] \times [0,n]^{d-2}$. Secondly, in the definition of $\kB_n$, we consider $\ell_i \subset \{i\} \times \{m\}^{d-3} \times [0,n]^2$  and $\ell'_i \subset \{i\} \times \{n-m\}^{d-3} \times [0,n]^2$, two bottom-top (resp.\ left-right) crossings in the last two coordinates when $i$ is odd (resp.\ even).  
\hfill $\square$

\section{Contact process on the key subgraph}
In this section, we  study the extinction time and the metastability of the  contact process on  key subgraphs defined in the previous section. 
\begin{lem} \label{tng} Let $\tau_{\ell,M}$ be the extinction time of the contact process on $\kC(\ell,M)$ starting from full occupancy. Then there exist positive constants $c$ and $K$ independent of $\lambda$, such that if $ \bar{\lambda}M \geq K $, then
\begin{align} \label{taun}
\pp \Big(\tau_{\ell,M} \geq \exp \big(c \ell M \log (\bar{\lambda} M)\big)\Big) \geq 1 - \exp\big(-c \ell M \log (\bar{\lambda} M)\big),
\end{align}
with $\bar{\lambda}= \lambda \wedge 1.$
\end{lem}
\begin{proof}
Let $(\xi_t)$ be the contact process on $\kC(\ell,M)$ with parameter $\lambda >0$. It is sufficient to consider the case  $\lambda \leq 1$, since the contact process is monotone in $\lambda$.
We assume also that $M \lambda \geq 640$. 

For $i \in [0, \ell]$, we say that  $i$ is  {\bf lit} at time $t$ (the term is taken from \cite{CD}) if the number of infected vertices in its attached complete graph at time $t$ is larger than $M/4$.   

 Let $T= \exp(M \log ( \lambda M) /16 )$.   For $r \geq 0 $ and $ i,j \in [ 0, \ell]  $ s.t. $ |i-j| = 1$ and $i+ r$ is even, we define 
\begin{align*}
Z^r_{i,j} = & 1( \{\textrm{$i$ is not lit at time } rT \}) \\
+ & 1( \{\textrm{$i$ is lit at time $rT$ and $i$ lights  $j$  at time $(r+1)T$} \}),
\end{align*}
where  "$i$ lights $j$  at time $(r+1)T $"   means that
\begin{align*}
& \left| \left \{  y \in C(j): \exists \, x \in C(i) \cap \xi_{rT}  \textrm{  s.t. } (x, rT) \longleftrightarrow (y,(r+1)T)  
  \textrm{ inside } C(i) \cup C(j) \cup \{i,j\}
 \right \} \right| \\
 & \geq  M/4,
\end{align*}
with $C(i)$ the complete graph attached at $i$.  Then $(Z^r_{i,j})$ naturally define an oriented percolation  by identifying 
\begin{align*}
\{Z^r_{i,j}=1\} \Leftrightarrow \{(i,r) \rightarrow (j,r+1)\}.
\end{align*}
It follows from Lemma \ref{l5} (ii) that
$$ \pn \left(Z^r_{i,j}=1 \, \big | \, \kF_{rT} \right) \geq 1- 5T^{-1} \quad \forall \, r \geq 0 \textrm{ and } |i-j| = 1,$$
where $\kF_t$ denotes the sigma-field generated by the contact process up to time $t$.

Moreover if $x \neq i$ and $y \neq j,$ then $Z^r_{x,y}$ is independent of $Z^r_{i,j}.$ Hence by a result of Liggett, Schonmann and Stacey \cite{LSS} (see also Theorem B26 in \cite{L})  the distribution of the family $(Z^r_{i,j})$ stochastically dominates the  measure of a  Bernoulli oriented percolation with parameter 
$$q \geq 1- T^{- \gamma},$$ 
with  $ \gamma \in (0,1)$. Moreover, if $\lambda M$  is large enough, then $1- T^{- \gamma} > 1 -\varepsilon $, with $\varepsilon$  as in Lemma  \ref{lc}.

 In summary, when $\lambda M $ is large enough, the distribution of $(Z^r_{i,j})$ stochastically dominates the one of an    oriented percolation on $[0, \ell] $ with density close to $1$. On the other hand, it follows from Lemma  $\ref{lc}$ (ii) that the oriented percolation process survives  up to the step 
$$\lf (1-q)^{-c \ell} \rf  \geq \lf T^{c\gamma \ell} \rf \geq \exp(c \gamma \ell M\log (\lambda M)),$$
 with probability larger than 
\[1-\exp(-c \gamma \ell M\log (\lambda M)),\]
for some constant $c>0$. Hence the result follows. 
\end{proof}
We now prove a  metastablity result  for connected graphs containing a copy of  $\kC(\ell,M)$. 
\begin{lem} \label{cvel}
Let $(G_n^0)$ be a sequence of connected graphs, such that $|G_n^0|\le n$, for all $n$. 
Let $\tau_n$ denote the extinction time of the contact process on $G_n^0$ 
starting from full occupancy. Assume that $G^0_n$ contains  a subgraph $H_n$, which is isomorphic to  $\kC(\ell_n,M)$. Then there exists a positive constant $K$, such that if $M \geq K/ (\lambda \wedge 1)$ and
\begin{align} \label{nas}
\frac{\ell_n}{d_n \vee \log n} \rightarrow \infty,
\end{align}
where $d_n = \max_{v \in G_n^0} d(v,H_n)$, then 
\begin{align*} 
\frac{\tau_n}{\en (\tau_n)}\quad  \mathop{\longrightarrow}^{(\kL)}_{n\to \infty} \quad  \kE(1).
\end{align*}
\end{lem}  
\begin{proof}
According to a result of Mountford   \cite[Proposition 1.2]{M}, it suffices to show that there exists a sequence $(a_n)$, such that $a_n=o(\en(\tau_n))$ and 
\begin{eqnarray}
\label{xivxi}
\sup_{v\in V_n}\, \pn(\xi^v_{a_n} \neq \xi_{a_n}, \xi^v_{a_n} \neq \varnothing) = o(1),
\end{eqnarray}
where $(\xi_t)_{t\ge 0}$ denotes the process starting from full occupancy.

Set $\bar{\lambda}= \lambda \wedge 1$. By  Lemma \ref{tng}, we get that if $\bar{\lambda}M$ is large enough, then 
\begin{align}
\label{taunDnmax}
\en(\tau_n) \geq \exp(c \ell_n M \log (\bar{\lambda}  M )), 
\end{align}
with $c$ as in this lemma. 
By \eqref{nas}, there is a sequence $(\varphi_n)$ tending to infinity, such that
\begin{align}
\label{Dnmaxdn}
\frac{\ell_n}{k_n} \rightarrow \infty,
\end{align}
with 
\begin{align*}
k_n = \lf (\log n \vee d_n)\varphi_n \rf.
\end{align*}
Now define 
\begin{align*}
b_n= s_{k_n} T \quad \textrm{and}\quad a_n=2b_n+1,  
\end{align*}
with $s_{k_n}$ as in Lemma \ref{lc} (iii) and $T= \exp( M \log(\bar{\lambda }M)/16)$.

Then \eqref{taunDnmax} and \eqref{Dnmaxdn} show that $a_n=o(\en(\tau_n))$, so it remains  to prove \eqref{xivxi} for this choice of $(a_n)$. To this end it is convenient to introduce the dual contact process. 
Given some positive real $t$ and 
$A$ a subset of the vertex set $V_n$ of $G_n$,  the dual process $(\hat{\xi}^{A,t}_s)_{s\le t}$ is defined by 
\[\hat{\xi}^{A,t}_s = \{ v\in V_n : (v,t-s)\longleftrightarrow A \times \{ t \} \},\]
for all $s\le t$. For any $v$, we have
\begin{eqnarray} \label{vkr}
&&\nonumber \pn(\xi^v_{a_n} \neq \xi_{a_n}, \xi^v_{a_n} \neq \varnothing)\\
 &=& \pn (\exists w\in V_n : \xi^v_{a_n}(w) = 0,\, \xi^v_{a_n} \neq \varnothing,\, \hat{\xi}^{w,a_n}_{a_n} \neq \varnothing) \notag\\
&\le & \sum_{w\in V_n} \pn\left(\xi^v_{a_n} \neq \varnothing,\, \hat{\xi}^{w,a_n}_{a_n} \neq \varnothing, \textrm{ and } \hat{\xi}^{w,a_n}_{a_n-t} \cap  \xi^v_t  = \varnothing \textrm{ for all } t\le a_n\right), 
\end{eqnarray}
So let us prove now that the last sum above tends to $0$ when $n\to \infty$.

\vspace{0.2 cm}
By the hypothesis, $G^0_n$ contains a subgraph $H_n$ which is  isomorphic to $\kC(k_n, M)$.  Hence, $H_n$ contains a chain of $k_n+1$ vertices  $x_0, \ldots, x_{k_n}$, such that $x_i$ is connected to $x_{i+1}$ for all $0 \leq i \leq k_n-1$. 
Moreover,  the vertex $x_i$ is attached a  complete graph of size $M$, say $C(x_i)$, for all $0 \leq i \leq k_n$. 

  Now we slightly change the definition of a lit vertex, and say that $x_i$ is lit if the number of its infected neighbors \textit{in $C(x_i)$} is larger than $M/4$ for $i= 0, \ldots, k_n$. 

We first claim that for any $v$
\begin{align}
\pp \big( \kA(v) ^c, \xi^v_{b_n} \neq \varnothing \big)  = o(1/n), \label{vbn1}
\end{align}  
where
\begin{align*}
\kA(v)= \Big \{\xi^v_{b_n} \neq \varnothing,  \,|\{i \in [(1- \beta)k_n/2,(1+ \beta)k_n/2 ]:  x_i \textrm{ is lit at time }b_n \}| \geq 3\beta k_n/4 \Big \},
\end{align*}
with $\beta$ as in Lemma \ref{lc}.

Suppose for a moment that \eqref{vbn1} holds. Then we also have 
\begin{align}
\pp \Big( \hat{\kA}(w) ^c, \hat{\xi}^{w, 2b_n+1}_{b_n} \neq \varnothing \Big)  = o(1/n), \label{vbn2}
\end{align}  
with 
\begin{align*}
\hat{\kA}(w) &= \Big \{\hat{\xi}^{w, 2b_n+1}_{b_n} \neq \varnothing, \, \, \exists \, S \subset [(1- \beta)k_n/2,(1+ \beta)k_n/2 ] \textrm{ with } |S| \geq 3 \beta k_n /4 \textrm{ and } \\
 &  W_i \subset C(x_i) \textrm{ with } |W_i| \geq M/4 \, \forall \, i \in S:  (x,b_n+1) \longleftrightarrow (w, 2b_n+1) \, \forall \, x \in \cup_{i \in S} W_i \Big  \}.
\end{align*}
Note that $\kA(v)$ and $\hat{\kA}(w)$ are independent for all $v$ and $w$. Moreover, on $\kA(v) \cap \hat{\kA}(w)$, there are more than $\beta k_n/2$ vertices which are lit in both the original and the dual processes. More precisely, there is a set $S \subset [(1- \beta)k_n/2,(1+ \beta)k_n/2 ]$ with $|S| \geq \beta k_n/2$ and sets $U_i, W_i \subset C(x_i) $ with $|U_i|, |W_i| \geq M/4$ for all $i \in S$, such that 
\begin{align*}
(v,0) \longleftrightarrow (x, b_n) &\quad \textrm{ for all } \quad x \in \cup_{i \in S} U_i \\ 
(y,b_n+1) \longleftrightarrow (w, 2b_n+1) &\quad \textrm{ for all } \quad y \in \cup_{i \in S} W_i.
\end{align*}
It is not difficult to show that there is a positive constant $c$, such that for any non-empty sets $U_i, W_i \subset C(x_i)$,  
\begin{eqnarray*} 
\pp  \left(U_i \times \{b_n\} \mathop{\longleftrightarrow}^{C(x_i)} W_i \times \{b_n+1\} \right) \geq c,
\end{eqnarray*} 
where the notation 
$$U_i \times \{b_n\} \mathop{\longleftrightarrow}^{C(x_i)} W_i \times \{b_n+1\}$$ 
means that there is an infection path inside $C(x_i)$ from a vertex in $U_i$ at time $b_n$ to a vertex in $W_i$ at time $b_n+1$. 

Moreover, conditionally on the sets $U_i,W_i$, these events are independent. Therefore, 
\begin{align*}
\pp \left(\exists i: U_i \times \{b_n\} \mathop{\longleftrightarrow}^{C(x_i)} W_i \times \{b_n+1\} \Bigm| U_i,W_i \right) \geq 1- (1-c)^{\beta k_n/2} = 1 -o(1/n),
\end{align*}
by our choice of $k_n$. This implies that 
\begin{align} \label{vbn3}
\pp \Big(\kA(v), \hat{\kA}(w), \hat{\xi}^{w,a_n}_{a_n-t} \cap  \xi^v_t  = \varnothing \textrm{ for all } t\le a_n \Big) = o(1/n).
\end{align}
Combining \eqref{vbn1}, \eqref{vbn2} and \eqref{vbn3} we obtain \eqref{vkr}. Hence, now it remains to prove \eqref{vbn1}. 

Fix a vertex $v \in G^0_n$.  We call $i_*$ an index, such that 
\begin{equation} \label{dvxi}
d(v,x_{i_*}) \leq d(v, H_n) +1 \leq d_n +1.
\end{equation}
As in Lemma \ref{tng}, we define an oriented percolation $(\tilde{\eta}_r)_{r \geq 0}$ on $[0,k_n]$ as follows. For $0 \leq i,j \leq k_n$ and $r \geq 0$, such that  $|i-j|=1$ and $i+r$ is even,  we let $Z^r_{i,j}=1$ (or equivalently $(i,r) \rightarrow (j,r+1)$) if either $x_i$ is not lit at time $rT$ or $x_i$ is lit at time $rT$ and $x_i$ lights $x_j$  at time $(r+1)T$.

As in Lemma \ref{tng}, there exists a positive constant $K$, such that if $\bar{\lambda} M\geq K$, then $(\tilde{\eta}_r)$ stochastically dominates a Bernoulli oriented percolation with parameter $1- \varepsilon$, with $\varepsilon$ as in Lemma \ref{lc}.
  
Assume that $d_n$ is even, if not we just take the smallest even integer larger than $d_n$. Then we set 
$$\tilde{d}_n= d_n + 2 k_n.$$
Now define for $k\geq 0,$
\begin{align*}
C_k=  \big \{  \exists \,  r,s \in [k \tilde{d}_n + d_n, (k+1)\tilde{d}_n ] \textrm{ s.t. } \tilde{\eta}^{i_*, k \tilde{d}_n + d_n}_r (0)=1, \tilde{\eta}^{i_*, k \tilde{d}_n + d_n}_s (k_n)=1 \big \},
\end{align*}
where for any $A \subset [0,k_n]$ and $t \geq s\geq 0$,
\begin{align*}
\tilde{\eta}^{A,s}_t = \{x \in [0,k_n]: \exists y \in A, (y,s) \rightarrow (x,t)\}.
\end{align*}
 Then using Lemma \ref{lc} (i), we get
\begin{align} \label{skj}
\pp \left(C_k \Bigm| \kF_{k \tilde{d}_n + d_n} \right) \geq c.
\end{align}
Using the same arguments for the claim (a) in Lemma \ref{lc}, we observe that on $C_k$, 
\begin{align} \label{ibx}
\tilde{\eta}^{{\bf 1}}_r = \tilde{\eta}^{i_*,k \tilde{d}_n + d_n }_r \quad \textrm{ for all } r \geq (k+1) \tilde{d}_n.
\end{align}
Define 
\begin{align*}
\kE= \Big \{ \big|\tilde{\eta}^{{\bf 1}}_{s_{k_n}} \cap [(1- \beta)k_n/2, (1+ \beta)k_n/2] \big| \geq 3 \beta k_n/4 \Big \},
\end{align*}
with $s_{k_n}$ as in Lemma \ref{lc} (iii). Using \eqref{ibx}, we get that on $C_k \cap \kE$, if $(k+1)\tilde{d}_n \leq s_{k_n}$ then
 \begin{align} \label{egs}
 \big|\tilde{\eta}^{i_*,k \tilde{d}_n + d_n }_{s_{k_n}} \cap [(1- \beta)k_n/2, (1+ \beta)k_n/2] \big| \geq 3 \beta k_n/4.
 \end{align}
Let  $K_n=\lf s_{k_n}/\tilde{d}_n \rf$ and for any $0 \leq k \leq K_n-1$, we define
\[A_k=\{\xi^v_{k\tilde{d}_n}\neq \varnothing\},\] 
and 
\begin{align*}
B_k = & \left\{ \xi_{k \tilde{d}_n}^v\times\{k\tilde{d}_n\} \rightarrow (x_{i_*},(k \tilde{d}_n + d_n -1)T) \right\} \cap \left \{x_{i_*} \textrm{ is lit at time } (k \tilde{d}_n + d_n )T \right\} \cap C_k.
\end{align*}
We have 
\begin{align} \label{xka}
\{\xi^v_{b_n} \neq \varnothing\} \subset \bigcap_{k=0}^{K_n-1} A_k.
\end{align}
On the other hand, if $x_{i_*}$ is lit at time $rT$ and $\tilde{\eta}^{i_*,r}_{s}(i)=1$ for $s>r$, then $x_i$ is lit at time $sT$. Hence by \eqref{egs} on $\kE$, if one of the events $(A_k \cap B_k)$ happens then $\kA(v)$ occurs.  Combing this with \eqref{xka}, we get
\begin{align}
\label{inc.bn}
\{\xi^v_{b_n} \neq \varnothing \} \cap \kA(v)^c \,\, \subset \,\,  \kE^c \cup  \left( \bigcap_{k=0}^{K_n-1} A_k \cap B_k^c \right).
\end{align} 
Using Lemma \ref{lc} (iii), we obtain a bound for the first term 
\begin{align} \label{pec}
\pp(\kE^c) \leq \exp(-c k_n) = o(1/n),
\end{align}
by the choice of $k_n$. For the second term,   by using \eqref{dvxi} and a similar argument as for \eqref{kii}, we have
\begin{eqnarray*}
\pn\left((v,t)\rightarrow (x_{i_*},t+(d_n-1)T)\right) \ge \exp(-C (d_n-1)T) \quad \textrm{for any  $t\ge 0$},
\end{eqnarray*}
for some constant $C>0$. On the other hand,    if $x_{i_*}$ is infected at time $t$ then it is lit at time $t+T$ with probability larger than $\exp(-CT)$.
Therefore combing with \eqref{skj}, we get that for any $k\le K_n-1$, 
$$\pn \left(B_k^c \Bigm| \kG_k \right){\bf 1}(A_k) \le 1-c\exp(-Cd_nT),$$
where  $\kG_k = \kF_{k \tilde{d}_n}$.  Iterating this, we get  
\begin{eqnarray} \label{abc}
\pn\left(\bigcap_{k=0}^{K_n-1} A_k \cap B_k^c\right) &\le  & (1-c\exp(-Cd_nT))^{K_n-1} = o(1/n),
\end{eqnarray}
where the last equality follows from the definition of $s_{k_n}$. Combining \eqref{inc.bn}, \eqref{pec} and \eqref{abc}
we get \eqref{vbn1} and finish the proof. 
\end{proof}
\section{Proof of Theorem \ref{trg}}
\subsection{Proof of (i)}
 We first prove the lower bound on $\tau_n$.  By Lemma \ref{p2},  there are positive constants $c $ and $K$, such that if $R^d \geq K/ (\lambda \wedge 1)$, then w.h.p.\ $G(n,R,g)$ contains  a subgraph $H_n$ which is isomorphic to $\kC(\ell_n, M)$, with $\ell_n= \lceil cnR^{-d} \rceil $ and $M= \lf c R^d \rf$. 
  
 If $\ell_n$ is bounded (or $R^d =\Theta( n)$), then $\kC(\ell_n,M)$ contains a complete graph of size of order $n$. Then Lemma \ref{l5} (i) implies that w.h.p.\ the extinction time is larger than $\exp(c n \log (\lambda n) )$, for some $c>0$.

If $\ell_n$ tends to infinity, then the result follows from Lemma \ref{tng}. 
  
   To prove the convergence in law of $\tau_n/\E(\tau_n)$, we recall some known results about the diameter of the giant component and the size of small components in RGGs. There is a positive constant $R_0$, such that if $R>R_0$, then w.h.p.
\begin{itemize}
\item[(a)] the diameter of the largest component is $D_n = \kO(n^{1/d}/R)$, \\
\item[(b)] the size of the second largest component is $\kO((\log n)^{d/(d-1)})$. 
\end{itemize} 
The first claim is proved by  Friedrich,  Sauerwald and  Stauffer in \cite[Corollary 6]{FSS}  and the second one is proved in Penrose's book \cite[Theorem 10.18]{P} when $g\equiv 1$. It is  not hard  to generalize these results for our model with $g$ bounded both from below and above. 

The second claim together with Lemma \ref{bgg} below show that w.h.p.\ the extinction time of the contact process on $G_n$ and on $G^0_n$ - the largest component - are equal. We are now in a position to complete the proof of (i). 
\begin{itemize}
\item[$\bullet$] If $R^d=o(n/ \log n)$, then 
\[\frac{nR^{-d}}{\log n} \rightarrow \infty \quad \textrm{and} \quad \frac{nR^{-d}}{D_n} \rightarrow \infty, \]
since by (a), $D_n = \kO((nR^{-d})^{1/d})$. On the other hand, $\ell_n = \Theta( nR^{-d})$ and $d_n = \max_{v \in G^0_n} d(v,H_n) \leq D_n$. Thus
\begin{align*}
\frac{\ell_n}{d_n \vee \log n} \rightarrow \infty.
\end{align*}
  Therefore, Lemma \ref{cvel} implies the convergence in law of $\tau_n/ \en(\tau_n)$. \\
\item[$\bullet$] If $n/ \log n = \kO(R^d)$, then $H_n$ contains a complete graph of size larger than $\sqrt{n}$. On the other hand, for all $k, \ell \in \mathbb{N}$ the complete graph of size $k \ell$ always contains a copy of $\kC(k,\ell)$. Hence, $G^0_n$ contains a copy of $\kC( \lf n^{1/4} \rf, \lf n^{1/4}\rf)$. We have 
\begin{align*}
\frac{\lf n^{1/4} \rf}{d_n \vee \log n} \rightarrow \infty,
\end{align*}
since 
\[d_n \leq D_n = \kO \Big(\left(nR^{-d}\right)^{1/d} \Big) =\kO \Big(\left( \log n \right)^{1/d}\Big).\]
Thus the result follows from Lemma \ref{cvel}.
\end{itemize}
  \hfill $\square$
\subsection{Proof of (ii)} We prove an upper bound on  the extinction time of the contact process on an arbitrary graph.
\begin{lem} \label{bgg}
Let $\tau_G$ be the extinction time of the contact process on a graph $G=(V,E)$ starting from full occupancy. Then
\begin{itemize}
\item[(a)] $\pp(\tau_G \leq F(|V|,|E|)) \geq 1- \exp(-|V|)$,\\
\item[(b)]$\en(\tau_G) \leq 2 F(|V|,|E|)$
\end{itemize}
with 
\begin{align*}
F(|V|,|E|)= |V| \left( 2+ \frac{4 \lambda |E|}{|V|} \right)^{|V|}.
\end{align*}
\end{lem}
\begin{proof}   Observe that (b) is a consequence of (a) and the following. For any $s>0$
\begin{align*}
\en(\tau_G) \leq \frac{s}{\pp(\tau_G \leq s)}.
\end{align*}
This result is  Lemma 4.5  in \cite{MMVY}. 

We now prove (a). Let us denote by $(\xi_t)$ the contact process on $G$ starting with full occupancy. By using  Markov's property and the monotonicity of the contact process, it suffices to show that
\begin{align} \label{tpm}
\pp(\xi_1 = \varnothing) \geq \exp(-|V| \log (2+ 4 \lambda |E|/ |V|)).
\end{align}
 Observe that the process dies at time $1$  if for any vertex $v$, it heals before $1$ and does not infect any neighbor.  Let $\sigma_v$ be the time of the first recovery at $v$, then $\sigma_v \sim \kE(1)$. Let $\sigma_{v \rightarrow}$ be the time of the first  infection spread from $v$ to one of its neighbors. Then it is the minimum of $\deg(v)$ i.i.d.\ exponential random variables with mean $\lambda$ and thus $\sigma_{v \rightarrow} \sim \kE(\lambda \deg(v))$. Moreover $\sigma_v$ and $\sigma_{v\rightarrow}$ are independent. Therefore 
\begin{align*}
\pp(\sigma_v < \min \{\sigma_{v \rightarrow}, 1\})= \frac{1- e^{-(1+ \lambda \deg(v))}}{1+ \lambda \deg(v)}  \geq \frac{1}{2(1 + \lambda \deg(v))}.
\end{align*}
On the other hand, these events $\{\sigma_v < \min \{\sigma_{v \rightarrow}, 1\}\}_v$ are independent. Then using Cauchy's inequality, we get that
\begin{align*} 
\pp(\xi_1 = \varnothing) & \geq \prod_{v \in V} (2 + 2\lambda \deg(v))^{-1} \\
 & \geq \left( \frac{2 |V| + 2 \lambda \sum_{v \in V} \deg(v)}{|V|} \right)^{-|V|} \notag \\
& =  \left( 2+ \frac{4 \lambda |E|}{|V|} \right)^{-|V|},
\end{align*}
which implies \eqref{tpm}.
\end{proof}
  \noindent {\it Proof of Theorem \ref{trg} (ii).} The upper bound on $\tau_n$ follows from Lemma \ref{bgg} and  the following: w.h.p.\  $G(n,R,g)=(V_n, E_n)$ with
\begin{itemize}
\item[$\bullet$] $ |V_n| \leq 2B n$\\
\item[$\bullet$] $|E_n| \leq C n R^d,$
for some $C=C(d,B)$.
\end{itemize}
The first claim is clear, since $|V_n|$ is a  Poisson random variable  with mean $$\int_{[0, \sqrt[d]{n}]^d} g(x) dx \leq B n.$$
For the second one, let $\ell = \lceil \sqrt[d]{n}/R \rceil$. We cover $[0, \sqrt[d]{n}]^d$ using translations by $R/2$ for each coordinate, accounting for $(2 \ell -1)^d$ boxes of volume $R^d$. We observe that   points at distance larger than $R$ are not connected. Hence, $|E_n|$ is less than the sum of the number of edges in the covering small boxes.

 These small boxes are partitioned into $2^d$ groups such that each group contains at most $\ell^d$ disjoint boxes with the same volume $R^d$. 
 
 The number of vertices in each box is stochastically dominated  by $Z$, a  Poisson random variable  with mean $BR^d$, since the integral of $g$ on a box is smaller than $BR^d$ (as $g(x) \leq B$ for all $x$). Hence, the number of edges  in a box is stochastically dominated by $Z^2$.
 
Moreover, in each group the numbers of edges are independent, as the boxes are disjoint. Therefore, using Chebyshev's inequality,  the total number of edges in a group is w.h.p.\ smaller than 
$$2 \ell^d \E(Z^2) = 2 \ell^d(BR^d)(BR^d+1).$$
Hence,  $|E_n|$ is w.h.p.\ less than 
$$2^{d+1} \ell^d (BR^d)(BR^d+1) \leq C nR^d,$$
for some $C=C(d,B)$ large enough. \hfill $\square$
\section{Some extensions}
\subsection{The  one-dimensional case} When $d=1$, RGGs are also called  random interval graphs, see for instance \cite{S}. We have the following result. 
\begin{prop} \label{pd1} Let $d=1$. Consider the contact process on  one-dimensional random geometric graphs $G(n,R,g)$ with  $g$ satisfying \eqref{cdg}. Then there exist positive constants $\varkappa, K, c$ and $C$ depending only on $b$ and $B$, such that the following statements hold.
\begin{itemize}
\item[(i)] If $R  \leq \varkappa \log n$, then w.h.p.\ the number of vertices in the largest component is $o(n^{2/3})$. Thus $\log \tau_n = o(n)$ w.h.p. \\
\item[(ii)] If $R \geq K \log n$, then w.h.p.\ the graph is connected and 
$$ c n \log (\lambda R) \leq \log ( \tau_n) \leq C n \log (\lambda R),$$
 and 
 \[ \frac{\tau_n}{\en(\tau_n)}\ \mathop{\longrightarrow}^{(\mathcal L)}_{n\to \infty}  \  \kE(1). \]
\end{itemize} 
\end{prop}
\begin{proof}
For (i), it is sufficient to consider  $R= \varkappa \log n$ with $\varkappa$ chosen later. We divide $[0,n]$ into $\lf n/R \rf$ intervals of length $R$, denoted by $I_1, \ldots, I_{\lf n/R \rf}$. Then the number of vertices in $I_i$ is a Poisson random variable with mean $\int_{I_i} g(x)dx = \Theta( R)$. Therefore  
\begin{align} \label{bd1}
\pp(\#\{\textrm{vertices in } I_i \} \leq R^2 \textrm{ for all } i = 1, \ldots, \lf n/R \rf) = 1-o(1),
\end{align}
as $R = \Theta( \log n)$. On the other hand, since $\int_{I_i} g(x) dx \leq BR$ for all $i$, the probability for an interval to be empty is larger than $e^{-BR}$. Hence
\begin{align} \label{bd2}
\pp \big(\textrm{there are at most $\lf\sqrt{n}\rf$ consecutive non-empty intervals} \big) &\geq 1- \lf n/R \rf (1-e^{-BR})^{\lf \sqrt{n}\rf} \notag \\
& = 1- o(1),
\end{align}
with $R = \varkappa \log n$ and $\varkappa$ small enough.

We observe that if an interval is empty, then there is no edge between vertices in the left-hand side and the right-hand side of this interval. Thus, it follows from  \eqref{bd1} and \eqref{bd2} that w.h.p.\ the number of vertices in any component is  smaller than $R^2 \lf \sqrt{n} \rf=o(n^{2/3})$.

We now prove (ii). If $R =\Theta( n)$, then the graph contains a complete graph of size of order $n$. Thus using Lemma \ref{l5} (i), we get the lower bound on  $\tau_n$. Assume that $R=o(n)$. We divide $[0,n]$ into $\lf 2n/R\rf$ intervals of length $R/2$, denoted by $J_1, \ldots, J_{[2n/R]}$. Then the numbers of vertices in  these intervals form a sequence of  independent Poisson random variables with mean larger than $bR/2$. For all $i \leq \lf 2n/R \rf$, we define 
\begin{align*}
\{J_i \textrm{ is good}\}= \{\textrm{the number of vertices in $J_i$ is larger than }  bR/4\}.
\end{align*} 
We have 
\begin{align*}
\pp(\poi(bR/2) \geq bR/4) \geq 1- \exp(cR),
\end{align*}
for some constant $c=c(b)>0$. Therefore
\begin{align} \label{ji}
\pp( J_i \textrm{ is good for all } i \leq  \lf 2n/R \rf) \geq 1- \lf 2n/R \rf e^{-cR} = 1-o(1),
\end{align}
with $R \geq K \log n$ and $K$ large enough. This implies that w.h.p.\ $G(n,R,g)$ contains as a subgraph a copy of $\kC(\lf 2n/R \rf, \lf bR/4 \rf)$ (note that the vertices in the same interval or in adjacent ones are connected). Thus similarly to  Theorem \ref{trg}, we get the lower bound on  $\tau_n$. The upper bound also follows from the same argument as in  Theorem \ref{trg}.    

For the connectivity, since all vertices in an interval $J_i$ or in adjacent ones are connected, we observe that 
\begin{align*}
\{J_i \textrm{ is good for all } i \leq \lf 2n/R \rf\} \subset \{J_i \textrm{ is non-empty for all } i \leq \lf 2n/R \rf \} \subset \{G_n \textrm{ is connected}\}.
\end{align*}
Therefore by \eqref{ji}, when $R \geq K \log n$ with $K$  large enough, w.h.p.\ $G_n$ is connected. In addition,  its diameter is $d_n \leq \lceil 2n/R \rceil$.

For the convergence in law of $\tau_n/\en(\tau_n)$, we note that  $\kC(\ell,M)$ always contains  a copy of $\kC(\ell \lf M/M_1 \rf, M_1)$   if $M\geq M_1$. Moreover, w.h.p.\ $G_n$ contains a copy of $\kC( \lf 2n/R \rf, \lf bR/4 \rf)$. Therefore, w.h.p.\ $G_n$ contains a copy of $\kC( \lf cn \rf,M)$ for some $c>0$ and $M$ as in Lemma \ref{cvel}. Hence, $G_n$ satisfies the hypothesis in Lemma \ref{cvel} and  the result follows.
\end{proof}
\subsection{An equivalent model} 
We consider another version of   random geometric graphs  with density function $f$ and connection radius  $r$, denoted by $G'(n,r,f)$. It is defined as follows: place independently $n$ points in $[0,1]^d$ according to $f$, then connect two points $u$ and $v$ by an edge if $\|u-v\| \leq r$. Suppose that 
\begin{align} \label{eqf}
0< b \leq f(x) \leq B< + \infty \quad \textrm{ for all }x. 
\end{align}
Similarly to the results for $G(n,R,g)$,  we have the following.
\begin{prop}
The results of Theorem \ref{trg} and Proposition \ref{pd1}  hold for the graph $G'(n,r,f)$ with $f$ satisfying \eqref{eqf} by replacing $R^d$ by $nr^d$ in the statements. 
\end{prop}
\begin{proof}
First, we observe that the law of a Poisson point process with intensity $g$ on a set $A$ conditionally  on its number vertices, say $N$, is the same as that  of the process defined by placing independently $N$ points in $A$ with density $g/ \int_A g$. 

 Therefore, the graph  $G(n,R,g)$ conditionally  on its size $|G|$ is isomorphic to  $G'(|G|,r,f)$ with  
\begin{align*}
r = R/\sqrt[d]{n} \qquad \textrm{ and } \qquad f(x) =g  \left( x \sqrt[d]{n}\right).
\end{align*}
To prove the lower bound on $\tau_n$, we consider $G_1 = G(n_1, R_1,g_1)$, where 
\begin{align*}
n_1= \lf n/(2B) \rf, \qquad R_1 = r \sqrt[d]{n_1} \qquad \textrm{ and } \qquad g_1(x) =f \left( \frac{x}{\sqrt[d]{n_1}}\right).
\end{align*}
Since $|G_1|$ is a Poisson random variable with mean less than $n/2$, w.h.p.\ $|G_1|$ is less than $ n$. Therefore w.h.p.\  $G_1$ can be coupled as a subgraph of  $G'(n, r, f )$. This domination together with Theorem \ref{trg}  and Proposition \ref{pd1}  imply the results for the lower bound on $\tau_n$ and the convergence in law of $\tau_n/\E(\tau_n)$ (note that the results of the connectivity and the diameter  of the largest component or the size of the second largest component also hold in  this model).

Similarly, for the upper bound on $\tau_n$, we consider  $G_2=G(n_2,R_2,g_2)$ with 
\begin{align*}
n_2= \lf 2n/b \rf, \qquad R_2 = r \sqrt[d]{n_2} \qquad \textrm{ and } \qquad g_2(x) =f \left( \frac{x}{\sqrt[d]{n_2}}\right).
\end{align*}
Then w.h.p.\ $G_2$ contains as a subgraph a copy of $G'(n,r,f)$. Thus by applying Theorem \ref{trg} and Proposition \ref{pd1}, we get desired results. 
\end{proof}
 \begin{ack} \emph{
 I am  grateful to Bruno Schapira for  many suggestions during the preparation of this work.  I would like also to thank the anonymous referee and an editor member for carefully reading this manuscript and many valuable comments. This work is supported by the Vietnam National Foundation for Science and Technology Development (NAFOSTED) under  Grant number 101.03--2017.07 }
 \end{ack}

\end{document}